\documentclass[11pt,twocolumn]{article}
\usepackage{amsthm}

\usepackage{graphicx}

\usepackage[english]{babel}
\usepackage{amssymb}
\usepackage{amsmath}
\usepackage{enumerate}
\usepackage{url}
\usepackage{color}
\usepackage{epsfig}
\usepackage{subfigure}
\usepackage{ifpdf}
 
\DeclareMathOperator    \conv           {conv}
\DeclareMathOperator    \cone           {cone}

\DeclareMathOperator    \proj           {proj}
\DeclareMathOperator    \rec                    {rec}

\DeclareMathOperator    \verts          {vert}

\newcommand{\bb}{\mathbb}
\newcommand{\old}[1]{{}}

\newcommand{\R}{\bb R}
\newcommand{\Q}{\bb Q}
\newcommand{\Z}{\bb Z}

\newtheorem{theorem}{Theorem}

\newtheorem{lemma}[theorem]{Lemma}

\newtheorem{remark}[theorem]{Remark}

\newtheorem{example}[theorem]{Example}

\renewcommand{\varepsilon}{\epsilon} 

\usepackage[normalem]{ulem}

\providecommand{\keywords}[1]{\textbf{\textit{Keywords---}} #1}

\begin{document}

\title{Note on the Complexity of the Mixed-Integer Hull of a Polyhedron}

\author{Robert Hildebrand%\thanks{robert.hildebrand@ifor.math.ethz.ch}
\and
Timm Oertel%\thanks{timm.oertel@ifor.math.ethz.ch} 
\and
Robert Weismantel%\thanks{robert.weismantel@ifor.math.ethz.ch}\\
%Institute for Operations Research, Department of Mathematics, ETH Z\"urich\\ R\"amistrasse 101, 8092 Z\"urich, Switzerland
}

\date{\today}

%\twocolumn[ 
\maketitle
  %\begin{@twocolumnfalse}
  
    \begin{abstract}
    We study the complexity of computing the mixed-integer hull $\conv(P\cap(\mathbb{Z}^n\times\mathbb{R}^d))$ of a polyhedron $P$.
Given an inequality description, with one integer variable, the mixed-integer hull can have exponentially many vertices and facets in $d$.  For $n,d$ fixed, we give an algorithm to find the mixed-integer hull in polynomial time. 
    Given a finite set $V \subseteq \Q^{n+d}$,  with $n$ fixed, we compute a vertex description of the mixed-integer hull of $\conv(V)$ in polynomial time and give bounds on the number of vertices of the mixed-integer hull.
    \end{abstract}
    
        \keywords{Mixed-integer hull, polyhedron, mixed-integer concave minimization}
        
 % \end{@twocolumnfalse}
  
%]

        \section{Introduction}

Given a rational polyhedron $P \subseteq \R^n \times \R^d$, we focus on computing the mixed-integer hull $P_{MI} = \conv(P \cap (\Z^n \times \R^d))$.   The mixed-integer hull is a fundamental object in mixed-integer linear programming and is well known to be a polyhedron~\cite{schrijver_theory_1986}.  
In 1992, Cook, Kannan, Hartman, and McDiarmid~\cite{cook_integer_1992} showed that the integer hull of $P_I = \conv( P \cap \Z^{n})$ has at most $2m^n(6n^2 \varphi)^{n-1}$ many vertices, where $m$ is the number of facets of $P$ and $\varphi$ is the maximum binary encoding size of a facet.  Hartman~\cite{hartmann-1989-thesis} gave a polynomial time algorithm in fixed-dimension to enumerate the vertices of the integer hull of a polyhedron.  
On the other hand,~\cite{lowerbound} provides examples where the number of vertices is as large as $c_n \varphi^{n-1}$, where $c_n$ is a constant depending only on $n$.  This indicates that the upper bounds are of the right order.
  See~\cite{zolotykh-2006} for a survey of these results and improvements, and also~\cite{charles-howe-king-2009} for a discussion of implementations. 
As far as we know, no similar results have been shown for the mixed-integer hull.

In Section~\ref{subsec:reduction-to-polytope}, we reduce the original task of computing the mixed-integer hull  of a polyhedron to the special case of polytopes by writing an extended formulation using a Minkowski-Weyl type decomposition of $P$.  Hence we proceed under the assumption that $P$ is bounded.
In the following sections we give two algorithms, based on different input descriptions, to compute the mixed-integer hull of a polytope; each algorithm produces a bound for the number of vertices of the mixed-integer hull.  

In Section~\ref{sec:two}, we consider $P$ given as an inequality description.  We show that even with one integer variable and the restriction of small subdeterminants of the inequality description, the mixed-integer hull can have exponentially many facets and vertices if $d$ varies.  Hence, we fix both the number of continuous and integer variables.  Through a simple scaling technique we can apply the results of~\cite{cook_integer_1992} and~\cite{hartmann-1989-thesis} to compute the integer hull of a scaled polyhedron.  By scaling back, we obtain the mixed-integer hull.  This leads to a bound on the number of vertices that is exponential in $n+d$.
  
In Section~\ref{sec:three}, we consider $P$ given by a list of vertices.  In this setting, we can allow $d$ to vary.  We show how to compute a vertex description of $P_{MI}$ in polynomial time in the encoding size of the vertices provided that the number of integer variables is fixed.  This algorithm implies a better bound on the number of vertices of the mixed-integer hull that depends on the number of vertices of the original polytope and does not depend on the number of continuous variables $d$. 
This algorithm also implies an algorithm for concave minimization over the mixed-integer points in a polytope since at least one solution lies at an extreme point.  
\begin{theorem}[Mixed-integer concave minimization]
Let $V \subseteq \Q^{n+d}$ be finite, $P = \conv(V)$, and $f\colon \R^{n+d} \to \R$ be concave.  When $n$ is fixed, the problem $\min\{ f(x,y) : (x,y) \in P \cap (\Z^n \times \R^d)\}$ can be solved in polynomial time in the evaluation time of $f$, $d$,$|V|$ and $\nu$, where $\nu$ is the maximum binary encoding size of a point in $V$.
\label{thm1}
\end{theorem}

Concave minimization over polyhedra presented by an inequality description is NP-Hard when the dimension varies, even for the case of minimizing a concave quadratic function over the continuous points in a cube.  This is because every extreme point can be a local minimum.  Most exact algorithms require in the worst case to enumerate all extreme points of the feasible region~\cite{concave-survey1986} which can be of exponential size in the number of inequalities.  In order to cope with these complexity results, it appears natural to assume that the underlying polyhedron is  presented by means of its vertices.  From the viewpoint of optimization, Theorem~\ref{thm1} is the mixed-integer analogue of integer concave minimization based on integer hull computations in fixed dimension for which vertex complexity and inequality complexity of the polyhedron are polynomial time equivalent.

\paragraph{Notation:}
For a set $Q \subseteq \R^n \times \R^d$, we define $Q_{\hat x} = \{(x,y) \in Q :  x = \hat x\} \subseteq \R^n \times \R^d$ and $\proj_x(Q) \subseteq \R^n$ as the projection of $Q$ onto the first $n$ variables.

%%%%%%%%
\section{Reduction to bounded polyhedra}
\label{subsec:reduction-to-polytope}
We begin by reducing the problem to finding the mixed-integer hull of a polytope.
We adapt a result of Nemhauser and Wolsey~\cite[Section I.4, Theorem 6.1]{Nemhauser-Wolsey-1988} to the mixed-integer case.  The proof is essentially the same, but we provide it here to obtain complexity bounds.
Recall that by the Minkowski-Weyl theorem, a polyhedron $P$ can be represented as an inequality description or an inner description, that is, $P = \{x \in \R^{n+d} : Ax \leq b\} = \conv(V) + \cone(W)$ for some finite sets $V \subseteq \Q^{n+d}$ and $W \subseteq \Z^{n+d}$.   We denote the recession cone of $P$ by $\rec(P) := \{x \in \R^{n+d} : Ax \leq 0\} = \cone(W)$.
In the following, we request that $P \cap (\Z^n \times \R^d)$ is non-empty, which can be tested in polynomial time provided that $n$ is fixed using Lenstra's algorithm~\cite{lenstra_integer_1983}.

\begin{lemma}[Relevant mixed-integer points]
\label{lem:reduce-to-polytope}
Let $P$ be a rational polyhedron given by the inner description $\conv(V) + \cone(W)$ for finite sets $V \subseteq \Q^{n+d}$ and $W \subseteq \Z^{n+d}$  (resp. an inequality description $\{ x \in \R^{n+d} : Ax \leq b\}$ for $A \in \Q^{m \times (n+d)}$ and $b \in \Q^m$).  Suppose that $P \cap (\Z^n \times \R^d) \neq \emptyset$.  In polynomial time in the input, we can compute a vertex description (resp. an inequality description) of a  rational polytope $Q$ of polynomial size such that 
$P = Q + \rec(P)$ and 
$$
P_{MI} =  \conv(Q \cap (\Z^n \times \R^d)) + \rec(P).
$$
\end{lemma}
\begin{proof}
By the Minkowski-Weyl theorem, we can assume that $P$ is given as $P = \conv(V) + \cone(W)$ for some finite subsets  $V\subseteq \Q^{n+d}$ and $W \subseteq \Z^{n+d}$.  
For any $x \in P \cap (\Z^n \times \R^d)$,  by this decomposition and Carath\'eodory's theorem we can write
\begin{align*}
x &= \sum_{i \in I} \lambda_i v_i + \sum_{j \in J} \mu_j w_j\\
 &= \sum_{i \in I} \lambda_i v_i + \sum_{j \in J} (\mu_j - \lfloor \mu_j \rfloor) w_j +  \sum_{j \in J} \lfloor \mu_j \rfloor w_j,
\end{align*}
where $\lambda_i, \mu_j \geq 0$, $\sum_{i \in I} \lambda_i = 1$, $v_i \in V$, $w_j \in W$, $|I| \leq n+d+1$ and $|J| \leq n+d$. 
Since $x \in \Z^n \times \R^d$ and  $\sum_{j \in J} \lfloor \mu_j \rfloor w_j \in \Z^{n\times d}$, we have that $x -  \sum_{j \in J} \lfloor \mu_j \rfloor w_j = \sum_{i \in I} \lambda_i v_i + \sum_{j \in J} (\mu_j - \lfloor \mu_j \rfloor) w_j  \in \Z^n \times \R^d$.

Therefore, if we define $T := \conv(V) + (n+d) \conv(W \cup \{0\})$, then $x -  \sum_{j \in J} \lfloor \mu_j \rfloor w_j   \in T \cap (\Z^n \times \R^d)$.  
It follows that 
$$ 
\conv(P \cap (\Z^n \times \R^d)) = \conv(T \cap (\Z^n \times \R^d)) + \rec(P).
$$
If $V$ and $W$ are given as input, then we are done by setting $Q = T$ for which we obtain a vertex description by taking the Minkowski sum $V + (n+d) W$.  On the other hand, if we are given as input an inequality description of $P$, then  the descriptions of $V$ and $W$ may be exponential in the input.  In this case, we instead determine a box that contains $T$ and intersect that box with $P$ to obtain our choice of $Q$.

More precisely, by for instance~\cite{schrijver_theory_1986}, we can choose $V$ and $W$ such that there exists an $R\geq 0$  of polynomial encoding size that is an upper bound on the infinity norm of the elements in $V$ and $W$. Setting $R' = (n+d + 1)R$, we have $T \subseteq B := [-R', R']^{n\times d}$.  Setting $Q = P \cap B$ finishes the argument.
\end{proof}

\section{Mixed-integer hull from inequalities}
\label{sec:two}

We study the problem of computing the mixed-integer hull of a polytope $P$ when we are given an inequality description of $P$.
Unfortunately, in this setting, it may be impossible to give a compact facet or vertex representation of the mixed-integer hull if $d$ is allowed to vary. This is demonstrated in the following example that has only one integer variable.  Notice that in the example, the maximum subdeterminant in absolute value of $A$ is two.  
   
\begin{example}
\label{example1}
Let $P = R \prod_{i=1}^{d+1} [-b_i,b_i] = \{x \in \R^{d+1} : -b_i \leq A_{i \cdot} x \leq b_i, \ \forall\ i\}$ be the linear transformation of the hypercube with 
$$
\begin{array}{cc}
R = 
\begin{bmatrix}
\tfrac12 &   \tfrac12    & \dots &\tfrac12\\
0 &  &  I_d
\end{bmatrix},
& 
A = R^{-1} =  
\begin{bmatrix}
2 & -1 & \dots & -1\\ 
0 &   & I_d
\end{bmatrix},
\end{array} 
$$
where $I_d$ is the $d\times d$ identity matrix.  Choosing odd numbers $b_i = 2^i + 1$,  the first coordinates of neighboring vertices of $P$ are contained in the interiors of non-adjacent unit intervals. It follows that the mixed-integer hull $P_{MI} = \conv( P \cap  \Z \times \R^d)$ does not contain any of the vertices of $P$.  In particular, it can be shown that every vertex of $P$ violates a distinct facet defining inequality of $P_{MI}$.  
  See Figure~\ref{fig:one}.
\end{example}
\begin{figure} 
\centering
\includegraphics[scale = .6]{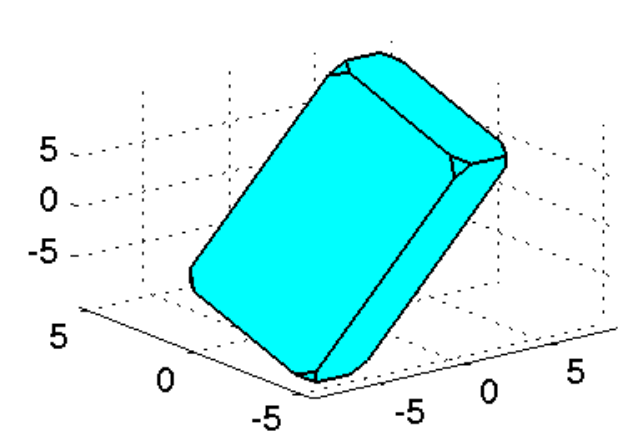} \ \ \ \ \includegraphics[scale = 0.6]{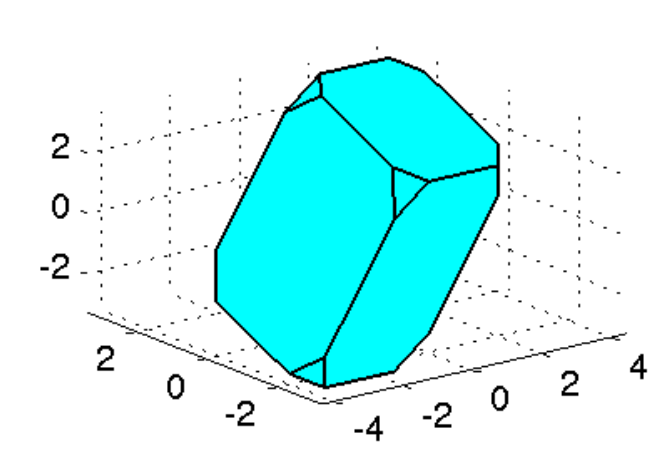}
\caption{ Left: The mixed-integer hull of the rotated cube in Example~\ref{example1} with $d +1 = 3$ and $b_i = 2^i + 1$. 
Right: The example with $b_i = 3$ mentioned in Remark~\ref{rem1}.  In both cases, every vertex of $P$ has been cut off, resulting in exponentially many facets and vertices.  }
\label{fig:one}
\end{figure}

\begin{remark}
\label{rem1}
In Example~\ref{example1}, if we instead choose $b_i = 3$, the mixed-integer hull $P_{MI}$ will still have exponentially many vertices and facets.  In this case though, the first coordinate of all vertices is polynomially bounded by $d$, which allows one to write an extended formulation of $P_{MI}$ as the convex hull of the union of polynomially many integer fibers.  By~\cite{Balas1983}, this leads to a polynomial size representation.  Note that it could still be possible that for the case of $b_i = 2^i + 1$ there is a polynomial size extended formulation for $P_{MI}$.
\end{remark}

In view of Example~\ref{example1}, we first restrict ourselves to the setting where $n,d$ are fixed.
We present a scaling algorithm that can be used to apply previous results on computing the integer hull.  This technique is similar to~\cite{cook-kannan-schrijver-1990} where they show that the mixed-integer split closure of a polyhedron is again a polyhedron by a scaling argument to transform the mixed-integer linear program to an integer linear program.

Let $P = \{(x,y) \in \R^n \times \R^d : (A_1, A_2) (x,y) \leq b \}$ be a polytope with $A_1 \in \Z^{m\times n}, A_2 \in \Z^{m\times d}$ and $ b \in \Z^{m}$.
We compute the mixed-integer hull $
P_{MI} = \conv(P \cap (\Z^n \times \R^d))
$
in time polynomial in the binary encoding size of $A_1,A_2$ and $b$. 
The output is either a facet or a vertex description of $P_{MI}$.  Since $n,d$ are both fixed, the facet and vertex descriptions are polynomial time equivalently computable (see, for instance,~\cite{schrijver_theory_1986}).

 For a polytope $Q\subseteq \R^n \times \R^d$, define  
$
Q^t := \{(x,t y) \in \R^n \times \R^d : (x,y)\in Q\}. 
$
Therefore, $Q^1 = Q$.
Notice that we can write
\begin{align*}
P^t &= \{(x,y) \in \R^n \times \R^d : (A_1, A_2) (x,\tfrac{1}{t} y) \leq b\}\\
&= \{(x,y) \in \R^n \times \R^d : (A_1, \tfrac{1}{t} A_2) (x,y) \leq  b\}.
\end{align*}
  
\begin{theorem}[Mixed-integer hull from inequality description]
\label{thm:alg1}
The mixed-integer hull $P_{MI}$ can be computed in time polynomial in the binary encoding size of $A_1, A_2, b$, provided that $n$ and $d$ are fixed.  Furthermore, let $\phi$ be the maximum binary encoding size of a row of $(A_1, A_2, b)$.  Then $|\verts(P_{MI})|\leq 2m^{n+d} (6(n+d)^2 \varphi)^{n+d-1}$ where $\varphi = \phi + n\phi(m+n)^{n+d}$.
\end{theorem}

\begin{proof}
We first show that we can compute in polynomial time an integer $t$ such that $P^t_I = P^t_{MI}$.
Let $(\hat x, \hat y)$ be a vertex of $P_{MI}$.     Notice that $(\hat x, \hat y)$ is a vertex of the $d$-dimensional polytope $P_{\hat x} = P \cap \{(x,y) : x = \hat x\}$. 
 We analyze the vertices of $P_{\hat x}$ by considering its inequality description as $P_{\hat x} = \{ (x,y) : A (x,y) \leq \bar b\}$ where
$$
A = \begin{bmatrix}
A_1 &  A_2 \\
 I_n & \mathcal O_{n\times d}\\
 -I_n & \mathcal O_{n \times d}\\
 \end{bmatrix}
  \in \Z^{(m + 2n) \times (n+d)}
  $$
  and $\bar b = (b, \hat x, -\hat x)$.  Here $\mathcal O_{n \times d}$ is the $n \times d$ matrix of all zeros.
  Therefore, there exists a basis $B$ of the rows of $A$ of size $n+d$ such that $(\hat x, \hat y) = (A_B)^{-1} (b, \hat x, -\hat x)_B$, where $A_B$ is the square submatix of $A$ with columns  indexed by $i \in B$.
By Cramer's rule, for any $i \in B$, we have
$$
(\hat x, \hat y)_i = \frac{\det(A_B^i)}{\det(A_B)}
$$
where $A_B^i$ is the matrix $A_B$ where the $i^{\text{th}}$ column is replaced by the vector $(b, \hat x, -\hat x)_B$.  
Therefore $\det(A_B^i) \in \Z$.  Hence
$
(\hat x, \hat y)_i  \cdot \det(A_B) \in \Z,
$
that is, $\hat y \cdot \det(A_B) \in \Z^d$ is integral.  Note that $\det(A_B)$ is completely independent of $(\hat x, \hat y)$.  Now if we let 
$$
t := \prod_{B \text{ basis}} \det(A_B),
$$
then $ (x,t y)$ is integral for any basis $B$ of $A$.
There are at most ${ m+n \choose n+d} \leq (m+n)^{n+d}$
 many bases $B$. By Hadamard's inequality, $\det(A_B)\leq \prod_{i=1}^n 2^\phi = 2^{n\phi}$.  Therefore, $t \leq (2^{n\phi})^{(m+n)^{(n+d)}} = 2^{n\phi (m+n)^{n+d}}$.

Since $P^t_{MI}$ is the convex hull of all polyhedra $P^t_{\hat x}$ for $\hat x \in \Z^n$, and $P^t_{\hat x}$ has integer vertices by our choice of $t$, we have that $P^t_{MI} = P^t_I$.
By computing the integer hull $P^t_I$ using~\cite{hartmann-1989-thesis}, and then scaling back by computing 
$(P^t_I)^{1/t} = (P^t_{MI})^{1/t} = P_{MI}$, we find a description of the mixed-integer hull.  Note that~\cite{hartmann-1989-thesis} yields a vertex description of the integer hull.  This can be converted to an inequality description in polynomial time since the dimension is fixed.
By~\cite{cook_integer_1992}, $|\verts(P_{MI})|  = |\verts(P^t_I)| \leq 2 m^{n+d}(6 (n+d)^2 \varphi)^{n+d-1}$ where $\varphi = \phi + n \phi (m+n)^{n+d}$.
\end{proof}

\section{Mixed-integer hull from vertices}
\label{sec:three}

We now consider the case where is given $P$ as a list of vertices.  In this setting, we show that the number of continuous variables $d$ may vary, and we still obtain a polynomial time algorithm to compute the mixed-integer hull.

We begin by showing how to compute the vertices of the integer hull of a polytope $Q = \conv(V)$ presented by its vertex set $V \subseteq \Q^n$.  To do so, we employ the algorithm from~\cite{hartmann-1989-thesis} to compute integer hulls from an inequality description, but this requires some care.  If we apply a standard transformation of $Q$ into an inequality description, this description could have many facets, causing the algorithm to require a double exponential complexity in terms of $n$.  Therefore, we instead find a triangulation of $V$ and then compute the integer hull of simplices, which have exactly $n+1$ facets, allowing us to procure a single exponential complexity in $n$ for the number of vertices of $Q_I$.  

\begin{lemma}[Integer hull from vertex description]
Let $V \subseteq \Q^n$ be finite and let $Q = \conv(V)$.  When $n$ is fixed, we can compute a vertex description of the integer hull $Q_I$ in polynomial time in $\nu$ and $|V|$ where $\nu$ is the maximum binary encoding size of a vector in $V$.  Furthermore, $|\verts(Q_I)| \leq 
\frac{1}{3} 12^n  n^{3n-2}\varphi^{n-1} |V|^{n+1}$
where $\varphi = 4 n^2 \nu$.  
If $|V| = n+1$, we obtain the tighter bound of $|\verts(Q_I)| \leq \frac{2}{3} 24^n n^{3n-2} \varphi^{n-1} $.
\label{lem:integer-hull-from-vertices}
\end{lemma}
\begin{proof}

Let $n' = \dim(\conv(V))$.  We compute a Delaunay triangulation of the points in $V$, yielding a list of $n'$-dimensional simplices.  As is well known, this can be done by computing the convex hull of the extended point set $V' = \{ (v, \| v\|_2^2) :  v \in V\}\subseteq \R^{n+1}$.  Then facets of $\conv(V')$ correspond to  $n'$-dimensional cells of the triangulation.   The convex hull can be computed in polynomial time using~\cite{avis-fukuda-1992}.  See for instance,~\cite{fortune1997} for a discussion of algorithms to compute the Delaunay triangulation.  Moreover, there are at most ${ |V| \choose n'+1 } \leq |V|^{n+1}$ cells.  If $|V| = n+1$, we can instead use the bound $2^{n+1}$ for the number of cells. For general $|V|$, a tighter asymptotic bound on the number of simplices is $O(|V|^{\lceil n/2\rceil})$~\cite{Seidel1995}.   

For each simplex $C$ in the triangulation, we compute an inequality description.  Since it may be lower dimensional, at most $2n$ inequalities are needed to describe it.  By~\cite[Theorem 10.2]{schrijver_theory_1986}, these inequalities can be computed in polynomial time and each inequality has a binary encoding size bounded by $\varphi = 4n^2 \nu$.

We can apply the result of~\cite{cook_integer_1992} to see $|\verts(C_I)| \leq 2m^n(6n^2 \varphi)^{n-1}$ where $m$ is the number of facets.  Here $m\leq 2n$, thus
$|\verts(C_I)| \leq 2(2n)^n(6n^2 \varphi)^{n-1}$. Applying this to each cell $C$, we have at most 
$ 
|V|^{n+1} 2(2n)^n (6n^2 \varphi)^{n-1} = \frac{1}{3} 12^n n^{3n-2}\varphi^{n-1}  |V|^{n+1}
$
 many vertices of the integer hull, or only $ \frac{2}{3} 24^n  n^{3n-2}\varphi^{n-1}$ if we choose $|V| = n+1$ and use the improved bound mentioned above.
By~\cite{hartmann-1989-thesis}, it follows that these can be computed in polynomial time when $n$ is fixed.  This creates a superset of the vertices of the integer hull; points that are not vertices can be discovered using linear programming, which can be done in polynomial time. 
\end{proof}  

Although this approach may not produce a tight bound on the number of vertices of the integer hull, the bound in Lemma~\ref{lem:integer-hull-from-vertices} is in the right order of magnitude in terms of $\varphi$.  Indeed, in~\cite{lowerbound}, they show that for every dimension $n\geq 2$, there exists a simplex $P\subseteq \R^n$ given as an inequality description with encoding size $\varphi$ such that $P_I$ has at least $c_n \varphi^{n-1}$ many vertices, where $c_n$ is a constant depending only on $n$.  This shows that even if $P$ has a small number of vertices, the number of vertices of the integer hull can be large.

  Our next goal is to make use of Lemma~\ref{lem:integer-hull-from-vertices} in order to compute the mixed-integer hull.  This requires one more ingredient.

\begin{lemma}
Every vertex of $P_{MI}$ lies in a face $F$ of $P$ with $\dim(F) \leq n$.  Furthermore, let $n' = \min(n, \dim(P))$ and let $\mathcal F_{n'}$ denote the set of faces of $P$ of dimension $n'$.  Then 
\begin{equation}
\verts(P_{MI}) \subseteq \bigcup_{F \in \mathcal F_{n'}} \bigcup_{\hat x \in \verts(\proj_x(F)_I)} \verts(F_{\hat x}).
\label{eq:MIvertices}
\end{equation}
\label{lem:dim-faces}
\end{lemma}
\begin{proof}
Since $P_{MI} = \conv(P_{\hat x} : \hat x \in \Z^n)$, every vertex $(\bar x, \bar y)$ of $P_{MI}$ is a vertex of $P_{\bar x}$.  Since $P_{\bar x} = \{ (x,y) : A(x,y) \leq b, I x = \bar x\}$ and any vertex of $P_{\bar x}$ is defined uniquely by $n+d$ tight inequality constraints, at least $d$ of those tight constraints come from the inequalities $A(x,y) \leq b$.  Hence, $(\bar x, \bar y)$ satisfies at least $d$ linearly independent tight constraints from $P$, i.e., it is contained in a face $F$ that is $n+d - d = n$ dimensional at most.   We only need maximal faces of this type; hence, we only consider $F$ of dimension $n' = \min(n, \dim(P))$. Furthermore, there exist  $d$ tight inequalities $(\bar A_1, \bar A_2) (x,y) \leq \bar b$ such that $\bar A_2$ is invertible.   
 
Then, for any $\hat x \in \proj_x(F)$, the corresponding $\hat y$ such that $(\hat x, \hat y)\in F$ is given uniquely as

$$
\hat y = \bar A_2^{-1} \bar b - \bar A_2^{-1} \bar A_1 \hat x.
$$ 

Now let $(\bar x, \bar y) \in \verts(P_{MI})$ and suppose that $\hat x \notin \verts(\proj_x(F)_I)$.  Then $\bar x = \sum \mu_i x^i$ for some $\mu_i > 0$, $\sum \mu_i = 1$, and $x^i \in \verts(\proj_x(F)_I)$.  Then setting, $y^i := \bar A_2^{-1} \bar b - \bar A_2^{-1} \bar A_1 x^i$ it follows that $y^i \in F\cap (\Z^n \times \R^d)$ and $(\bar x, \bar y) = \sum \mu_i (x^i, y^i)$; therefore, $(\bar x, \bar y)$ is not a vertex of $P_{MI}$.
This proves~\eqref{eq:MIvertices}.
\end{proof}

\begin{theorem}[Mixed-integer hull from vertex description]
\label{thm:vertex-description}
Let $V \subseteq \Q^{n+d}$ be finite and let $P = \conv(V)$.  For fixed $n\geq 1$ 
there exists an algorithm to compute $\verts(P_{MI})$ that runs in polynomial time in $d,\nu, |V|$ where $\nu$ is the maximum binary encoding size of a point in $V$.  Furthermore, 
$$
|\verts(P_{MI})| \leq 
\tfrac{4}{3} 48^n n^{3n-2}\varphi^{n-1} |V|^{n+1}
$$
where $\varphi = 4n^2 \nu$. 
\end{theorem}
\begin{proof}
Let $n' = \min(n, \dim(P))$ and let $(\hat x, \hat y)$ be a vertex of $P_{MI}$.  Note that $n' \leq |V|$ since for any polytope $P$ it holds that $\dim(P) \leq  |\verts(P)| -1$.   By Lemma~\ref{lem:dim-faces}, there exists a face $F$ of $P$ of dimension $n'$ such that $(\hat x, \hat y) \in \verts(F_{\hat x})$ and $\hat x \in \verts(\proj_x(F)_I)$.  By Carath\'eodory's theorem, there exist $n'+1$ vertices of $F$ such that $(\hat x, \hat y) \in \bar F = \conv(v^1, \dots, v^{n'+1})$.   Since $\bar F \subseteq F$ and $(\hat x, \hat y) \in \bar F$, it follows that $\hat x \in \verts(\proj_x(\bar F)_I)$ and $(\hat x, \hat y) \in \verts(\bar F_{\hat x})$.  Indeed, if $\hat x \notin \verts(\proj_x(\bar F)_I)$, then it can be written as a strict convex combination of points in $\verts(\proj_x(\bar F)_I) \subseteq \conv(\verts(\proj_x( F)_I))$, which shows that $\hat x \notin \verts(\proj_x(F)_I)$.

Therefore, we can compute a superset of the vertices of $P_{MI}$ by enumerating every $(n'+1)$-elementary subset $\{v^1, \dots, v^{n'+1}\}$ of $V$, yielding at most ${ |V| \choose n'+1} \leq |V|^{n'+1}$ sets to consider.
Fix a subset $\{v^1, \dots, v^{n'+1}\}$ and set $\bar F = \conv(v^1, \dots, v^{n'+1})$.  We will next find a vertex description of $\conv(\bar F \cap (\Z^n \times \R^d))$.  

To this end, let $\hat v^i = \proj_x(v^i)$.
Applying Lemma~\ref{lem:integer-hull-from-vertices} with $Q = \conv(\hat v^1, \dots, \hat v^{n'+1})$, we have a vertex description of $Q_I$ with at most 
$
\frac{2}{3} 24^{n'} \varphi^{n'-1} {n'}^{3{n'}-2}
$
many vertices.
For each vertex $\hat x$ of $Q_I$, we compute the vertices of $\bar F_{\hat x}$, which can be written as the set of $(x,y) \in \R^n \times \R^d$ satisfying
\begin{equation*}
 (x,y) = \sum_{i=1}^{n'+1} \lambda_i v^i,   x = \hat x, \sum_{i=1}^{n'+1} \lambda_i = 1, \lambda_i \geq 0.
\end{equation*}
Every vertex of this set corresponds to a unique face of the $(n'+1)$-dimensional standard simplex.  There are $2^{n'+1}$ many faces of the $n'$-dimensional simplex.  Hence, this set has at most $2^{n'+1}$ vertices, which can be enumerated, for instance by~\cite{avis-fukuda-1992} in time polynomial in the encoding size $\varphi$ provided that $n$ is fixed.  This number of enumerated points provides us with an upper bound on the number of vertices of $P_{MI}$.  Since  we can test whether points are vertices of $P_{MI}$ using linear programming, the proof is complete.      
\end{proof}

Combined with Lemma~\ref{lem:reduce-to-polytope}, the above theorem gives a similar result for polyhedra.  
\begin{remark} 
When the number of vertices of $P$ is bounded by a polynomial in $d$, Theorem~\ref{thm:vertex-description} can be applied and leads to a polynomial time algorithm.  For example, consider a polyhedron $P = \{(x,y) \in \R^{n+d} : A(x,y) \leq b, (x,y) \geq 0\}$, where $A \in \Q^{m \times (n+d)}$, $b \in \Q^m$.  Even with $m$ fixed, the number of inequalities grows with the dimension because of the nonnegativity constraints.  By counting basic solutions, it follows that $|\verts(P)| \leq 
{ n+d+m \choose n+d } = { n+d+m \choose m} \leq (n+d+m)^m$.  Therefore, for fixed $m$ and $n$, using Theorem~\ref{thm:vertex-description}, we can find a vertex description of the mixed-integer hull in polynomial time.  This applies, for instance, to a mixed-integer knapsack problem with no upper bounds on the variables.  
\end{remark} 

\bibliographystyle{abbrv}
\bibliography{references.bib}

\end{document}